\documentclass[twocolumn]{article}
\usepackage[a4paper, margin=1.8cm]{geometry}

\usepackage{amsthm}
\usepackage{cite}
\usepackage{amsmath,amssymb,amsfonts}
\usepackage{algorithmic}
\usepackage{graphicx}
\usepackage{algorithm,algorithmic}
\usepackage{hyperref}

\usepackage{mathrsfs,mathtools}
\usepackage{tikz}
\usetikzlibrary{intersections, calc, arrows.meta,bending}
\usepackage[svgnames]{xcolor}

\usepackage{comment}

\newtheorem{theorem}{Theorem}
\newtheorem{lemma}{Lemma}

\newtheorem{proposition}{Proposition}

\newtheorem{assumption}{Assumption}
\newtheorem{remark}{Remark}

\newtheorem{result}{Result}

\newcommand{\sto}{ \rm{s.t.} }

\newcommand{\norm}[1]{\left\lVert #1 \right\rVert}
\newcommand{\R}{\mathbb{R}}
\newcommand{\xstar}{x^\star}
\newcommand{\etastar}{\eta^\star}
\newcommand{\lstar}{\lambda^\star}

\definecolor{green2}{rgb}{0.2, 0.75, 0.2}

\definecolor{orange2}{RGB}{255,165,0}
\definecolor{violet2}{rgb}{0.93, 0.51, 0.93}
\definecolor{green2}{rgb}{0.2, 0.75, 0.2}

\usetikzlibrary{plotmarks, fit}

\newcommand{\tikzref}[1][]{%
    \begin{tikzpicture}[baseline,yshift=0.3em]
        \draw [mark repeat=2,mark phase=2,#1]
        plot coordinates {
            (0cm,0cm)
            (0.25cm,0cm)
            (0.5cm,0cm)
        }; \node[fit=(current bounding box),inner xsep=0.1em]{};
    \end{tikzpicture}%
}

\begin{document}

\title{Global Convergence of Control-Based Lagrangian Flows\\for Non-Convex Optimization}
\date{}

\author{S. Pirrera\thanks{Corresponding author: email \texttt{simone.pirrera@polito.it}}, F. Ripa, D. Astolfi, V. Cerone, S. M. Fosson and D. Regruto\thanks{This is the extended version of the paper published as: 
S. Pirrera, F. Ripa, D. Astolfi, V. Cerone, S. M. Fosson, and D. Regruto, 
``Global Convergence of Control-Based Lagrangian Flows for Non-Convex Optimization,'' 
IEEE Control Systems Letters, 2026, doi: 10.1109/LCSYS.2026.3709153.}~\thanks{S. Pirrera, F. Ripa, V. Cerone, S. M. Fosson and D. Regruto are with the Dipartimento di Automatica e Informatica,
Politecnico di Torino; email: \texttt{\{name.surname\}@polito.it}.\\D. Astolfi is with Universit\'e Claude Bernard Lyon 1, CNRS, LAGEPP UMR 5007, 43 boulevard du 11 novembre 1918, F-69100, Villeurbanne, France; email: \texttt{daniele.astolfi@univ-lyon1.fr}.}~\thanks{S. Pirrera's work received funding from the European Union’s Horizon Europe programme under the Marie Sk\l{}odowska-Curie grant agreement No. 101276493.\\Research partially supported by ANR via ALLIGATOR project (ANR-22-CE48-0009-01).}
}

\maketitle

\begin{abstract}
    This paper studies the continuous-time dynamics generated by control-theoretic Lagrangian methods for equality-constrained optimization. In particular, we consider dynamics induced by proportional-integral and feedback linearization controllers, which have recently been proposed as alternatives to primal-dual gradient methods. Unlike global convergence results for these dynamics, which rely on strong convexity of the objective function or boundedness assumptions, we exploit the geometric structure induced by the constraints. Specifically, we show global exponential convergence for non-convex problems that satisfy a suitable convexity property when restricted to the constraint manifold.
\end{abstract}

\section{Introduction}\label{sec:IN}
Constrained optimization problems arise in a wide range of engineering applications. 
For instance, they are central in supply chain management~\cite{garcia2015supply}, machine learning~\cite{gallego2024constrained}, and system identification~\cite{nlid_arxiv25}.
In most cases, explicit solutions are not available, which motivates the use of iterative algorithms that drive the optimization variables toward the problem optimum.
Optimization algorithms are defined as discrete-time dynamical systems or are obtained by discretization of continuous-time ones~\cite{bin2024semiglobal,kelly2024interconnected}. The continuous-time analysis provides valuable insights into the algorithm's properties (such as stability, performance, and robustness), and can guide the design of novel discrete-time implementations~\cite{su2016differential}.A widely studied class of optimization dynamics is that of Lagrangian and primal-dual methods, which are derived from Lagrange multiplier theory. Among these, the primal-dual gradient dynamics (PDGD), has been extensively studied; see, e.g.,~\cite{fei10,cherukuri2016asymptotic,qu19}.

An alternative perspective is to use control-theoretic tools to design optimization dynamics with rigorous convergence and performance guarantees. This approach has been applied to unconstrained optimization via integral quadratic constraints~\cite{les16,scherer2025tutorial}. In constrained optimization, the controller design regulates Lagrange multipliers to drive the system to an equilibrium satisfying first-order optimality conditions. We refer to the flows generated by the resulting dynamics as control-based Lagrangian flows.

Among different control strategies,~\cite{cmo24} and~\cite{cerone_feedback_2024}
propose the use of proportional-integral (PI) and feedback linearization (FL) controllers, while~\cite{allibhoy_control-barrier-function-based_2024} employs control barrier functions
and \cite{hauswirth2020differentiability} investigates the use of anti-windup. In the context of time-varying optimization, the control approach has been adopted in~\cite{casti_control_2025} via robust control and in~\cite{astolfi_repetitive_2025} using repetitive control tools. Many studies leverage convexity assumptions to simplify convergence analysis and establish global results~\cite{zhang_constrained_2026}.

This work focuses on the dynamics arising from PI and FL control. The PI dynamics is studied in~\cite{cmo24} and~\cite{allibhoy_anytime_2025}, where global convergence is established under the assumption of a strongly convex objective function. Regarding the FL dynamics,~\cite{cmo24} proves global convergence for strongly convex objective functions and local stability of the minima in the non-convex case, while the recent contribution~\cite{zhang_constrained_2026} establishes global asymptotic convergence for non-convex problems under suitable boundedness assumptions. 

We aim to relax these assumptions by exploiting the geometric structure induced by the equality constraints. In particular, we establish global exponential convergence for a class of problems in which the objective function is not necessarily convex or bounded on the Euclidean space, but satisfies a suitable convexity property on the constraint manifold. This is the ambient-space counterpart of geodesic strong convexity, a notion central to Riemannian optimization~\cite{zhang2016first,boumal2023introduction}. 

Such a property appears in the analysis of optimization methods defined on Riemannian manifolds. Methods including Riemannian gradient descent guarantee convergence when the objective is geodesically convex on the manifold.
However, they operate directly on the manifold and do not naturally yield globally defined dynamical systems in Euclidean space. Similar convexity conditions are also encountered in the analysis of augmented Lagrangian methods (ALM), but are usually restricted to a neighborhood of the optimal solution in the case of nonlinear constraints; see, e.g.,~\cite{bertsekas1996constrained}.

\subsubsection*{Outline}
Section \ref{sec:PS} presents the problem formulation. Sections~\ref{sec:FL} and \ref{sec:PIalg} establish the global exponential convergence of the FL and PI dynamics, respectively, under the assumption of convexity on the constraint manifold. In Section \ref{sec:further_comments}, we investigate the relationship between PI, FL, ALM, and PDGD dynamics. Section \ref{sec:academic_es} presents a numerical example. Section \ref{sec:CON} concludes the paper.

\subsubsection*{Notation}
Given $x\in \R^n$, we denote by $\|x\|$ its Euclidean norm. Given a function $g: \R^n \to \R^m$, we denote by $\nabla_x^\top g(x) \in \R^{m \times n}$ the Jacobian matrix of $g$ and with $\nabla_x g(x)$ its transpose. Given a matrix $A\in \R^{n\times m}$, $\|A\|$ is the matrix norm induced by the Euclidean norm.

\section{Problem statement and Background}
\label{sec:PS}
\allowdisplaybreaks
Let $f: \R^n \to \R$ and $h : \R^n \to \R^m$ be smooth functions. We consider the equality-constrained optimization problem
\begin{equation} \label{eq:opt} \begin{aligned}
    & \min_{x \in \R^n} ~ f(x) \qquad {\sto} \quad  h(x) = 0.
\end{aligned} \end{equation}

\noindent Throughout this paper, we make the following assumption.
\begin{assumption}\label{ass:Jrank}
    The Jacobian matrix $\nabla_x^\top h(x) \in \mathbb{R}^{m \times n}$ is of full row rank and there exists a real constant $\underline{m} >0$ such that $\nabla_x^\top h(x) \nabla_x h(x) \succeq \underline{m} I$ for all $x \in \R^n$.
\end{assumption}
\noindent   
Assumption~\ref{ass:Jrank} is common in the constrained optimization literature; see, e.g., \cite{nocedal2006numerical,bin2024semiglobal,sch00}. It is met in many applications, including a broad class of optimal control and system identification problems~\cite{nlid_arxiv25}.

In this paper, we consider the problem of analyzing a class of continuous-time dynamical systems whose state trajectory converges to points $\xstar$ that satisfy the first-order optimality conditions of Problem~\eqref{eq:opt}. 
Specifically, we seek points $\xstar$ that satisfy the first-order optimality conditions, summarized by the following result.

\begin{result}[First-order conditions~\cite{lue_book,ber99,nocedal2006numerical}]
\label{lagr_th}
Let $\xstar \in \R^n$ be a local minimum of Problem~\eqref{eq:opt}, and let Assumption~\ref{ass:Jrank} hold. Then, there exists a unique Lagrange multiplier vector $\lstar \in \R^m$ such that
\begin{align}\label{FO}
    & \nabla_x f(\xstar) + \nabla_x h(\xstar)\lstar = 0, \quad h(\xstar)=0.
\end{align}
\end{result}

The continuous-time dynamics PDGD~\cite{fei10}, described by
\begin{equation}\label{eq:pdgd}
    \dot x = -\nabla_x f(x) - \nabla_x h(x) z, \qquad  \dot z = h(x),
\end{equation}
is widely adopted to compute the stationary points of Problem~\eqref{eq:opt}. However, global exponential convergence of~\eqref{eq:pdgd} is guaranteed only under the assumption that $f(x)$ is strongly convex on the entire ambient space $\R^n$ (see, e.g., \cite{qu19,davydov2025time,bin2024semiglobal}), while its trajectories may exhibit limit cycles \cite{lin2020gradient} or instability 
\cite{cmo24} (see also Sec.~\ref{sec:academic_es}) in non-convex settings.

From a control-theoretic perspective, \eqref{eq:pdgd} represents a feedback loop coupling the plant introduced in~\cite{cmo24}
\begin{equation}\label{eq:cmo_plant} 
    \dot x = -\nabla_x f(x) - \nabla_x h(x) \lambda, \qquad y = h(x),
\end{equation} 
with an integral controller driven by the constraint violation. Here, $\lambda \in \R^m$ is the control input and $y \in \R^m$ is the output; the integrator dynamics of $z$ guarantee that any closed-loop equilibrium satisfies the feasibility condition $h(x) = 0$.

In this paper, we show that the control-based continuous-time algorithms introduced in~\cite{cmo24} allow us to relax the requirement of global convexity on $f(x)$. Our global convergence analysis leverages a structural property of the optimization problem: there exists a suitable change of variables that removes the constraints and yields an equivalent strongly convex optimization problem.

To analyze the system dynamics within the constraint-induced 
coordinates, we introduce a global change of variables. Specifically, 
we suppose that there exists a smooth map $q: \mathbb{R}^n \to \mathbb{R}^{n-m}$ 
satisfying 
$Q(x) \doteq \nabla_x^\top q(x)\,\nabla_x q(x) \succeq \underline{q}\, I$, 
for some $\underline{q} > 0$ and all $x \in \R^n$, such that the map
\begin{equation} \label{eq:diffeo}
x \mapsto \xi = \begin{bmatrix} y \\ \eta \end{bmatrix} 
\doteq \Phi(x) = \begin{bmatrix} h(x) \\ q(x) \end{bmatrix},
\end{equation}
with $y \in \R^m$ and $\eta \in \R^{n-m}$, is a global diffeomorphism 
of $\R^n$. Within this setting, we formalize our structural convexity 
assumption on Problem~\eqref{eq:opt}.

\begin{assumption}\label{ass:convexity}
The feasible set $\Omega = \{x \in \R^n : h(x)=0\}$ is connected, and the $\tilde f: \R^{n-m} \to \R$, defined by $\tilde f(\eta) \doteq f(\Phi^{-1}(0,\eta))$, is $\rho_\eta$-strongly convex on $\R^{n-m}$. That is, there exists a constant $\rho_\eta > 0$ such that, for all $\eta_1, \eta_2 \in \R^{n-m}$,
\begin{equation*}
\tilde f(\eta_2) \ge \tilde f(\eta_1) 
+ \nabla_\eta^\top \tilde f(\eta_1)\, (\eta_2 - \eta_1) 
+ \frac{\rho_\eta}{2}\|\eta_2 - \eta_1\|^2.
\end{equation*}
\end{assumption}
\begin{remark}\label{rem:convex_ass_generality}
    Assumption~\ref{ass:convexity} is weaker than  
    convexity of Problem~\eqref{eq:opt} (i.e., convexity of $f(x)$ in $\mathbb{R}^n$ and affine constraints).  Under this assumption, Problem~\eqref{eq:opt} admits a unique global optimizer $x^\star \in \R^n$ satisfying the conditions in Result~\ref{lagr_th}.
\end{remark}

\section{Stability analysis of FL dynamics}\label{sec:FL}
This section uses Assumption~\ref{ass:convexity} to study the stability of the FL dynamics proposed in~\cite{cmo24}, i.e.,
\begin{subequations}\label{eq:fl-cmo-general}\begin{align}
    \dot x & = -\nabla_x f(x) - \nabla_x h(x) \lambda\\
    \lambda & = [\nabla_x^\top  h(x) \nabla_x h(x)]^{-1}(-\nabla_x^\top h(x) \nabla_x f(x) + \mathcal{G}(y)), \label{eq:fl-cmo_lambda}
\end{align}\end{subequations}
where $\mathcal{G}: \R^{m} \rightarrow \R^{m}$ can be designed so that $\dot y = \mathcal{G}(y)$ is globally exponentially stable. 

Under the diffeomorphism~\eqref{eq:diffeo}, system \eqref{eq:cmo_plant} takes the form: 
\begin{subequations} \label{eq:normalform} \begin{align}
    \dot y & =\! \nabla_x^\top h(x) \dot x =\! - \nabla_x^\top h(x) \nabla_x f(x)\! -\! \nabla_x^\top h(x) \nabla_x h(x) \lambda \\
    \dot \eta &= \nabla_x^\top q(x) \dot x \!= -\nabla_x^\top q(x) \nabla_x f(x). \label{eq:zerodyn} 
\end{align}\end{subequations}
Eq.~\eqref{eq:zerodyn} follows from the fact that it is always possible to select $q(x)$ such that its Jacobian $\nabla_x^\top q(x)$ satisfies $\nabla_x^\top q(x) \nabla_x h(x)=0$ for all $x\in \R^n$, thus making the dynamics independent of the input $\lambda$, as shown in~\cite{isi95}. The following result holds.

\begin{theorem}[Zero dynamics stability]
\label{th:flcmo_ges}
    Under Assumption~\ref{ass:convexity}, $\eta^\star = q(x^\star)$ is a global exponentially stable equilibrium for the zero dynamics associated with~\eqref{eq:normalform}, i.e., there exists a real constant $c_\eta \geq 1$ such that, for all $t\geq0$ and all $\eta(0)\in \R^{n-m}$,
    \begin{equation}
        \norm{\eta(t)-\eta^\star} \leq c_\eta e^{- \underline{q}} \rho_\eta t \norm{\eta(0)-\eta^\star}.
    \end{equation}
\end{theorem}

\begin{proof}
    Define $p(\eta) = \Phi^{-1}(0,\eta)$ and
    $\tilde f(\eta) = f(p(\eta))$.
    We show that the zero dynamics of~\eqref{eq:normalform}, i.e., \begin{equation}\label{eq:zd_explicit}
        \dot\eta = -\nabla_x^\top q(x)\,\nabla_x f(x), \qquad x = p(\eta),
    \end{equation}
    is related to the negative gradient flow via $\dot\eta = -Q(p(\eta))\nabla_\eta \tilde f(\eta)$. 
    Then, since $Q(x) \succeq \underline{q} I$, Assumption~\ref{ass:convexity} implies global exponential stability
    of such a gradient flow with rate $\underline{q}\rho_\eta$ (see, e.g.,~\cite{bullo2022contraction}).
    We evaluate the gradient of $\tilde f(\eta) = f(p(\eta))$ by applying the chain rule: \begin{equation}\label{eq:chain_rule}
        \nabla_\eta \tilde f(\eta)
        = \nabla_\eta p(\eta)\,\nabla_x f(x)\big|_{x=p(\eta)}
    \end{equation}
    Let us differentiate $p(\eta)=\Phi^{-1}(0,\eta)$ with respect to $\eta$:    \begin{equation}\label{eq:th1_rowselect}
        \nabla_\eta^\top p(\eta)
        = \nabla_\xi^\top \Phi^{-1}(\xi)\big|_{\xi=(0,\eta)}
          \begin{bmatrix} 0 \\ I \end{bmatrix}.
    \end{equation}
    By the inverse function theorem applied to the global diffeomorphism $\Phi$, we obtain    
    \begin{align}\label{eq:th1_inv_f_th}
        & \nabla_\xi^\top \Phi^{-1}(\xi)
        = \bigl(\nabla_x^\top \Phi(x)\bigr)^{-1}
        = \begin{bmatrix} \nabla_x^\top h(x) \\ \nabla_x^\top q(x) \end{bmatrix}^{-1} = \\
        & ~ =
        \begin{bmatrix}
            \nabla_x h(x)\,(\nabla_x^\top h(x)\,\nabla_x h(x))^{-1} & \nabla_x q(x)\,Q(x)^{-1}
        \end{bmatrix}, \notag
    \end{align}
    where the last equality uses the orthogonality condition $\nabla_x^\top q(x)\,\nabla_x h(x)=0$, the definition of $Q(x)$, and the block-inverse formula.
    Combining Eqs.~\eqref{eq:th1_rowselect} and \eqref{eq:th1_inv_f_th} yields
    \begin{equation}\label{eq:grad_p}
        \nabla_\eta^\top p(\eta) = \nabla_x q(x)\,Q(x)^{-1}\big|_{x=p(\eta)}
    \end{equation}
    Finally, substituting~\eqref{eq:grad_p} into~\eqref{eq:chain_rule},
    \begin{equation}
        \nabla_\eta \tilde f(\eta)
        = Q(x)^{-1} \,\nabla_x^\top q(x)\,\nabla_x f(x)\big|_{x=p(\eta)}
    \end{equation}
    so~\eqref{eq:zd_explicit} reads $\dot\eta = -Q(p(\eta))\nabla_\eta\tilde f(\eta)$, concluding the proof.
\end{proof}

Using FL theory~\cite{isi95} and Theorem~\ref{th:flcmo_ges}, we can prove global exponential convergence of the FL dynamics in Eq.~\eqref{eq:fl-cmo-general} under the following additional assumptions.

\begin{assumption}\label{ass:G_ges}
    The origin of $\dot y = \mathcal{G}(y)$ is globally exponentially stable, i.e., there exist  $c_g \geq 1,\rho_g >0$ such that 
    \begin{equation}
        \norm{y(t)} \leq c_g e^{-\rho_g t} \norm{y(0)}, 
        \quad \forall y(0)\in \R^m, \forall\,  t\geq0. 
    \end{equation} 
\end{assumption}
Note that $\mathcal{G}$ can always be designed to satisfy this condition.

\begin{assumption}\label{ass:Lip_diffeo}
    The global diffeomorphism defined in Eq.~\eqref{eq:diffeo} and its inverse, $\Phi^{-1}: \R^n \rightarrow \R^n$, are globally Lipschitz with constants $L_\Phi$ and $L_{\Psi}$, respectively.
\end{assumption}

\begin{assumption}\label{ass:L1}
    Let $\tilde x = \Phi^{-1}(y,\eta)$ and $p = \Phi^{-1}(0,\eta)$. The function $\nabla_x^\top q(x) \nabla f(x)$ $: \R^n \rightarrow \R^{n-m}$ satisfies the following Lipschitz property for some real $L_1>0:$
\begin{equation}\label{eq:transverse_bound}
        \Big\|
        \nabla_x^\top q(\tilde x) \nabla_x f(\tilde x)
        -
        \nabla_x^\top q(p) \nabla_x f(p)
        \Big\|
        \le
        L_1 \norm{y}.
    \end{equation}
    \end{assumption}
    
\begin{theorem}\label{corollary:flcmo}
    Let Assumptions~\ref{ass:Jrank}--\ref{ass:L1} hold. Then, the FL 
    dynamics in Eq.~\eqref{eq:fl-cmo-general} is globally exponentially 
    stable.
    Specifically, for any rate $\rho_1 < \min\{\underline{q}\rho_\eta, \rho_g\}$, there exists a real
    $c_1 \geq 1$ such that,
    for all $x(0) \in \R^n$ and $t \geq 0$,
    \begin{equation}\label{eq:flcmo_bound}
        \|x(t)-\xstar\| \leq c_1\, e^{-\rho_1 t}\, \|x(0)-\xstar\|.
    \end{equation}
\end{theorem}

\begin{proof}
    By Assumption~\ref{ass:Lip_diffeo}, since $x = \Phi^{-1}(y,\eta)$ and $\xstar = \Phi^{-1}(0,\eta^\star)$, we can bound the state error in terms of the normal form coordinates:
   \begin{equation}\label{eq:bound_x}
        \|x-\xstar\|
        \leq L_\Psi \left\|\begin{bmatrix} y \\ \eta-\eta^\star \end{bmatrix}\right\| \leq L_\Psi\bigl(\|y\| + \|\eta-\eta^\star\|\bigr) .
    \end{equation}
    Thus, it suffices to establish exponential decay of both $\|y(t)\|$ and $\|\eta(t)-\eta^\star\|$.
    First, by Assumption~\ref{ass:G_ges}, we have $\|y(t)\| \leq c_g e^{-\rho_g t} \|y(0)\|$.
    Next, we analyze the dynamics of $\eta$, which can be viewed as the zero dynamics of~\eqref{eq:cmo_plant} perturbed by a vanishing term $\delta(t)$, i.e.,
    \begin{equation}\label{eq:eta_perturbed}
        \dot\eta = -Q(p) \nabla_\eta \tilde f(\eta) + \delta(t),
    \end{equation}
    where $\tilde f(\eta)$ is defined in Theorem~\ref{th:flcmo_ges}, and the perturbation term is given by $\delta(t) \doteq \nabla_x^\top q(p) \nabla_x f(p) - \nabla_x^\top q(x) \nabla_x f(x)$ with $p = \Phi^{-1}(0,\eta)$.
    To establish the exponential convergence of~\eqref{eq:eta_perturbed}, we prove that it is input-to-state stable (ISS) with respect to the input $\delta(t)$, which decays exponentially to zero.
    Consider the ISS Lyapunov candidate $V(\eta) = \tfrac{1}{2}\|\eta-\eta^\star\|^2$. Using \eqref{eq:eta_perturbed}:
    \begin{equation}
        \dot V = -(\eta-\eta^\star)^\top Q(p) \nabla_\eta \tilde f(\eta) + (\eta-\eta^\star)^\top \delta(t).
    \end{equation}
    Recalling $\tilde f$ is $\rho_\eta$-strongly convex and that $\nabla_\eta \tilde f(\eta^\star) = 0$, and since $Q(p) \succeq \underline{q} I$, we use the strong monotonicity property $-(\eta-\eta^\star)^\top Q(p) \nabla_\eta \tilde f(\eta) \leq -\underline{q}\rho_\eta \|\eta-\eta^\star\|^2 = -2\underline{q}\rho_\eta V$ to bound the first term.
    Next, applying in sequence the Cauchy--Schwarz and Young's inequalities to the second term, we obtain:
    \begin{subequations}        
    \begin{align}
        \dot V &\leq -2\underline{q}\rho_\eta V + \|\eta-\eta^\star\|\|\delta(t)\|
                \label{eq:dotV_initial}\\
        & \leq (-2\underline{q}\rho_\eta+\epsilon) V + \frac{1}{2\epsilon}\|\delta(t)\|^2 \label{eq:dotV_iss}
    \end{align}
    \end{subequations}
    for any $\epsilon \in (0, 2\underline{q}\rho_\eta)$, which is the required ISS dissipation inequality;
    see~\cite{sontag2013mathematical}.
    For any desired convergence rate $\rho_1 < \min\{\underline{q}\rho_\eta, \rho_g\}$, selecting $\epsilon$ sufficiently small ensures that $2\underline{q}\rho_\eta - \epsilon > 2\rho_1$ and $2\rho_g > 2\rho_1$.
    By Assumptions~\ref{ass:G_ges} and~\ref{ass:L1}, $\delta(t)$ is exponentially bounded as:
    \begin{equation}\label{eq:th2_bnd_delta_v2}
        \|\delta(t)\|^2 \leq L_1^2 \|y(t)\|^2 \leq L_1^2 c_g^2 e^{-2\rho_g t} \|y(0)\|^2.
    \end{equation}
    Plugging Eq.~\eqref{eq:th2_bnd_delta_v2} in Eq.~\eqref{eq:dotV_iss} and applying the comparison lemma yields $V(t) \leq e^{-2\rho_1 t} \left( V(0) + M_\delta \|y(0)\|^2 \right),$ for some positive constant $M_\delta$.
    Taking the square root and using $\|\eta(t)-\eta^\star\| = \sqrt{2V(t)}$ we obtain the bound:
    \begin{equation}\label{eq:eta_bound_final}
        \|\eta(t)-\eta^\star\|
        \leq e^{-\rho_1 t}\! \left( \|\eta(0)-\eta^\star\| \!+ \!\sqrt{2M_\delta} \|y(0)\| \right).
    \end{equation}
    Finally, using the Lipschitz bounds from Assumption~\ref{ass:Lip_diffeo} on the initial conditions, $\|y(0)\| \leq L_h\|x(0)-\xstar\|$ and $\|\eta(0)-\eta^\star\| \leq L_q\|x(0)-\xstar\|$, we substitute~\eqref{eq:eta_bound_final} and the bound on $\|y(t)\|$ back into~\eqref{eq:bound_x}.
    This yields the desired result with $c_1 = L_\Psi (L_q + \sqrt{2M_\delta}L_h + c_g L_h)$.
\end{proof}

\begin{remark}
    Theorem~\ref{corollary:flcmo} quantifies the convergence rate of the FL dynamics directly in terms of the problem data. In particular, we establish that $\rho_1$ 
    is determined by the minimum of the two rates associated with the zero dynamics and the externally assigned y-dynamics induced by the choice of $\mathcal{G}$.
\end{remark}
A closely related analysis is addressed in the independent work~\cite{zhang_constrained_2026}.  While both works share similar technical assumptions regarding gradient boundedness (see Appendix~II), they differ fundamentally in their core premises and resulting guarantees. Specifically, \cite{zhang_constrained_2026} relies on the assumption that $f(x)$ is lower-bounded and establishes asymptotic convergence to a first-order KKT point, which in the non-convex setting, might be a local minimum. In contrast, our work leverages the structural property in Assumption~\ref{ass:convexity}. This distinction is crucial, as it enables our framework to guarantee global exponential convergence to the unique global minimum. An illustrative example of a problem that fulfills Assumption~\ref{ass:convexity}, but $f(x)$ is not lower-bounded, is detailed in Sec.~\ref{sec:academic_es}.

\section{Stability analysis of PI dynamics} \label{sec:PIalg}
This section studies the convergence of the PI dynamics proposed in~\cite{cmo24} and described by the following system
\begin{align}\label{eq:picmo_zform} 
    \dot x\! =\! -\nabla_x f(x) - \nabla_x h(x) \left(k_p h(x) + k_i z \right), \quad    \dot z\! = \!h(x),
\end{align} 
under Assumption~\ref{ass:convexity}. 

We conduct this analysis by studying the system in the normal-form coordinates introduced in Sec.~\ref{sec:FL}. We start by rewriting \eqref{eq:picmo_zform} as:
\begin{subequations}\label{eq:picmo_lamb_form} \begin{align*}
    \dot x\! &=\! -\nabla_x f(x) - \nabla_x h(x) \lambda \\
    \dot \lambda\! &=\! - k_p\left( \nabla_x^\top h(x) \nabla_x h(x) \lambda\!+ \!\nabla_x^\top h(x) \nabla_x f(x)\! + \!k h(x) \right),
\end{align*} \end{subequations}
where $k = k_i/k_p \in \R$.
Next, we consider the global diffeomorphism defined by Eq.~\eqref{eq:diffeo} and the coordinate transformation $w \doteq \lambda - k_p  y - \lstar$. Using these coordinates, we obtain 
\begin{equation}\label{eq:picmo_zd_form}\begin{aligned}
\dot{\eta} &= -\nabla_x^\top q(x) \nabla_x f(x)\\
\dot{y} &= -\nabla_x^\top h(x) \nabla_x f(x)\! -\! \nabla_x^\top h(x) \nabla_x h(x) (w + \lstar)+ \\
&\quad -\! k_p \nabla_x^\top h(x) \nabla_x h(x) y\\
\dot{w} &= k_p \,k \, y. 
\end{aligned}\end{equation}

We establish the convergence of \eqref{eq:picmo_zform} using Assumption~\ref{ass:convexity} and the following additional assumption.

\begin{assumption}\label{ass:lip_l}
The function $\varphi(x): \R^n \rightarrow \R^m$ defined by
\begin{equation}
    \varphi(x) \doteq -(\nabla_x^\top h(x) \nabla_x h(x))^{-1}(\nabla_x^\top h(x) \nabla_x f(x))
\end{equation}
is globally Lipschitz, i.e., there exists a real $L_2>0$ such that $\norm{ \varphi(x_1) - \varphi(x_2) } \leq L_2 \norm{x_1-x_2}$ for all $x_1,x_2 \in \R^n$.
\end{assumption}

\noindent The following lemma holds.

\begin{lemma}[Lipschitz continuity of the nonlinear term $r$]\label{lemm:bnd_r}
    Let $r(\eta,y): \R^{n-m} \times \R^m \rightarrow \R^m$ be defined by: 
    \begin{equation}
        r(\eta,y) \doteq \nabla_x^\top h(\tilde x) \nabla_x f(\tilde x)+\nabla_x^\top h(\tilde x) \nabla_x h(\tilde x) \lstar
    \end{equation}
    with $\tilde x = \Phi^{-1}([y^\top,\eta^\top]^\top)$. Then, there exist real constants $\ell_r^\eta,\ell_r^y>0$ such that, for all $y \in \R^m$ and $\eta \in \R^{n-m}$, it holds
    \begin{equation}\label{eq:r_bound}
    \norm{r(\eta,y)}
    \le
    \ell_r^\eta \norm{\eta-\etastar}
    +
    \ell_r^y \norm{y}.
    \end{equation}
\end{lemma}
\begin{proof} Using $\lstar = \varphi(\xstar)$ (see, e.g.,~\cite{ber99}) and the definition of $\varphi$ in Assumption~\ref{ass:lip_l}, the residual can be rewritten as $r(\eta,y) = \nabla_x^\top h(\tilde x) \nabla_x h(\tilde x)\, [\varphi(\xstar) - \varphi(\tilde x)]$; then Eq.~\eqref{eq:r_bound} follows from Assumptions~\ref{ass:Lip_diffeo} and~\ref{ass:lip_l}. \end{proof}
\noindent Using Lemma~\ref{lemm:bnd_r}, we can prove the main result of this section.
\begin{theorem}\label{th:picmo_ges}
Let Assumptions~\ref{ass:Jrank}, \ref{ass:convexity}, and~\ref{ass:Lip_diffeo}-\ref{ass:lip_l} hold. There exists a constant $k_p^\star$ such that, if $k_p > k_p^\star >0$ and $0< k_i < \underline{m} k_p^2$, then $(\xstar,\lstar)$ is a globally exponentially stable equilibrium of the dynamics~\eqref{eq:picmo_zform}, i.e., there exist real constants $c_\pi \geq 1,\rho_\pi>0$ such that
\begin{equation}
    \norm{\begin{bmatrix}
        x(t)-\xstar \\ \lambda(t)-\lstar
    \end{bmatrix}} \leq c_\pi e^{-\rho_\pi t}\norm{\begin{bmatrix}
        x(0)-\xstar \\ \lambda(0)-\lstar
    \end{bmatrix}}. 
\end{equation}
\end{theorem}

\begin{proof} Define $\tilde \eta = \eta - \eta^\star$ and $\zeta \doteq [y^\top, w^\top]^\top$.
Consider the candidate Lyapunov function $V(\tilde\eta,\zeta) = \mu V_1(\tilde\eta) +  V_2(\zeta)$, where $V_1 = \frac12 \norm{\tilde \eta}^2$ and
\begin{equation}
    V_2(\zeta) =\zeta^\top \Pi \zeta, \qquad \Pi = \begin{bmatrix}
    I & k_p^{-1}I\\
    k_p^{-1}I & \theta k_p^{-1} I
\end{bmatrix}
\end{equation}
for some $\mu>0$ and $\theta>0$ to be chosen.
Note that $\Pi$ is positive definite for any $k_p >1/\theta$ (this can be shown using the Schur complement with respect to the upper-left block). 
Hence, $V$ is a valid quadratic Lyapunov function for the state $(\tilde\eta^\top, \zeta^\top )^\top$.
Now let $\tilde x = \Phi^{-1}(y,\eta)$ and $p = \Phi^{-1}(0,\eta)$. Adding and subtracting $\nabla_x^\top q(p) \nabla_x f(p)$, we obtain
\begin{equation}
\label{eq:V1dot_split}
\begin{aligned}
    \dot V_1 &= -\tilde \eta^\top \big( \nabla_x^\top q(p) \nabla_x f(p) \big) + \\
    & \quad + \tilde \eta^\top \left( \nabla_x^\top q(p) \nabla_x f(p) - \nabla_x^\top q(\tilde x) \nabla_x f(\tilde x) \right). 
\end{aligned} 
\end{equation}
From Theorem~\ref{th:flcmo_ges}, the term $-\tilde \eta^\top ( \nabla_x^\top q(p) \nabla_x f(p))$ is bounded by $-\underline{q}\rho_\eta \norm{\tilde\eta}^2$, while, by Assumption~\ref{ass:L1}, the term $\tilde \eta^\top \left( \nabla_x^\top q(p) \nabla_x f(p) - \nabla_x^\top q(\tilde x) \nabla_x f(\tilde x) \right)$ is bounded by $L_1 \|\tilde \eta \|\, \|y\|$. Using Young's inequality, we obtain
\begin{equation}\label{eq:V1_bound}
\dot V_1
\le
-\frac{1}{2} \underline{q} \rho_\eta \norm{\tilde \eta}^2
+
\frac{L_1^2}{2\underline{q}\rho_\eta}\norm{y}^2.
\end{equation}
Define $H(\eta,y) \doteq \nabla_x^\top h(\tilde x) \nabla_x h(\tilde x)$. 
Then, \eqref{eq:picmo_zd_form} is recast to
\begin{equation}\label{eq:y_compact}
\begin{aligned}
    \dot y &= -r(\eta,y)-H(\eta,y)w-k_p H(\eta,y)y, \\
 \dot w &= k_p k y,
\end{aligned}
\end{equation}
with $r$ defined in Lemma~\ref{lemm:bnd_r}.
In compact form, we obtain
\[
\dot \zeta
=
\underbrace{\begin{bmatrix}
-k_p H(\eta,y) & -H(\eta,y)\\[1mm]
k_p k I & 0
\end{bmatrix}}_{N(\eta,y)} \,\zeta
+
\begin{bmatrix}
-r(\eta,y)\\
0
\end{bmatrix}.
\]
Therefore, the time derivative of $V_2$ is
\begin{equation*}\label{eq:Vf_dot_general}
\dot V_2
=
-\zeta^\top \Psi(\eta,y) \zeta
- 2 \zeta^\top \Pi \begin{bmatrix} r(\eta,y) \\ 0 \end{bmatrix},
\end{equation*}
where $\Psi(\eta,y) \doteq - \big(\Pi N(\eta,y)+ N(\eta,y)^\top\Pi \big)$. By explicitly computing the matrix products, we have 
\begin{equation*}
    \Psi(\eta,y)
    = 
    \begin{bmatrix}
    2k_p  H(\eta,y)  - 2kI & 2H(\eta,y) - \theta k I\\[1mm]
 \star & 2{k_p^{-1}}H(\eta,y)
    \end{bmatrix} 
\end{equation*}
Equivalently, $\Psi(\eta,y) = 2 D \overline \Psi D$, with
$$
D= \begin{bmatrix}
    \sqrt{k_p}\,I & 0 \\ 0 &   \tfrac{1}{\sqrt{k_p}} I 
\end{bmatrix}, \quad
\overline\Psi = 
\begin{bmatrix}
      H  - \tfrac{k}{k_p}I & H - \tfrac{1}{2} \theta k I\\[1mm]
  \star & H
    \end{bmatrix},
$$
where we omitted the argument $(\eta,y)$ in $H$ for brevity.
To guarantee that $\overline\Psi \succ 0$, we apply the Schur complement with respect to the lower-right block $H \succeq \underline{m}I \succ 0$. The Schur complement $S$ is
\begin{align*}
    S &= (H - \tfrac{k}{k_p} I) - \big(H - \tfrac{1}{2}\theta k I\big) H^{-1} \big(H - \tfrac{1}{2}\theta k I\big) \\
    &= H - \tfrac{k}{k_p} I - \left(H - \theta k I + \tfrac{\theta^2 k^2}{4} H^{-1}\right) \\
    &= k\left(\theta - \tfrac{1}{k_p} \right)I - \tfrac{\theta^2 k^2}{4} H^{-1}.
\end{align*}
Since $H^{-1} \preceq {\underline{m}}^{-1} I$, we can lower-bound $S$ as
\begin{equation*}
    S \succeq \left[ k\left(\theta - \frac{1}{k_p}\right) - \frac{\theta^2 k^2}{4\underline{m}} \right] I  \succeq \left( \underline{m} - \frac{k}{k_p} \right) I,
\end{equation*}
where we selected $\theta = {2\underline{m}}/{k}$.
By choosing $k_p > {k}/{\underline{m}} > {1}/{\theta} = {k}/({2\underline{m}})$, which is equivalent to the condition $k_i < \underline m k_p^2$, we obtain $S \succeq c_s I \succ 0$ with $c_s = \underline{m} - {k}/{k_p} > 0$. Notice that this combined choice of $\theta$ and $k_p$ also ensures that $\Pi \succ 0$. 
Consequently, there exists a constant $\overline{c}_H > 0$ such that $\overline\Psi \succeq \overline{c}_H I$. Using $\Psi(\eta,y) = 2 D \overline \Psi D$, this implies
\begin{equation}\label{eq:Psi_bound}
    \Psi \succeq 2 \overline{c}_H D^2 = 2 \overline{c}_H \begin{bmatrix}
        k_p I & 0 \\ 0 & k_p^{-1} I
    \end{bmatrix}.
\end{equation}
Therefore, the quadratic term is bounded as:
\begin{equation}
    -\zeta^\top \Psi \zeta \le -2 \overline{c}_H k_p \norm{y}^2 - \frac{2 \overline{c}_H}{k_p} \norm{w}^2.
\end{equation}
Next, we bound the cross-term 
\begin{equation}
    -2\zeta^\top \Pi \begin{bmatrix} r(\eta,y) \\ 0 \end{bmatrix} = -2 y^\top r(\eta,y) - \frac{2}{k_p} w^\top r(\eta,y).
\end{equation}
Recalling from Lemma~\ref{lemm:bnd_r} that $r(\eta,y)$ is globally Lipschitz, i.e., $\norm{r(\eta,y)} \le \ell_r^\eta \norm{\tilde \eta} + \ell_r^y \norm{y}$, we can bound the terms using Young's inequality:
$$
\begin{aligned}
    \left|2 y^\top r(\eta,y)\right|    & \leq (2\ell_r^y + \ell_r^\eta) \norm{y}^2 + \ell_r^\eta  \norm{\tilde \eta}^2, \\
     \frac{2}{k_p} \left| w^\top r(\eta,y)\right| &\leq \frac{2(\ell_r^y)^2}{k_p\bar c_H}  \norm{y}^2 + 
     \frac{\bar c_H}{k_p}  \norm{w}^2
    +\frac{2(\ell_r^\eta)^2}{k_p\bar c_H}  \norm{\tilde \eta}^2.
\end{aligned}
$$
Combining all the previous bounds and assuming $k_p\geq1$, we obtain the following bound for $\dot V_2$:
\begin{equation}
    \dot V_2 \leq -  (2 \overline{c}_H k_p - C_y  ) \norm{y}^2 - \frac{\overline{c}_H}{k_p} \norm{w}^2 + C_\eta \norm{\tilde \eta}^2,
\end{equation}
where $C_y \doteq 2\ell_r^y + \ell_r^\eta +  \frac{2(\ell_r^y)^2}{\bar c_H}$ and $C_\eta \doteq \ell_r^\eta+\frac{2(\ell_r^\eta)^2}{\bar c_H}$. Consequently, the derivative of the Lyapunov function satisfies 
$$
\begin{aligned}
    \dot V   & \leq 
-  \left(2 \overline{c}_H k_p - C_y -  \mu \frac{L_1^2}{2\underline{q}\rho_\eta} \right) \norm{y}^2 
-  \frac{\overline{c}_H}{k_p} \norm{w}^2 + \\ 
& \quad - \left( \frac{\mu \underline{q}\rho_\eta}{2} - C_\eta \right) \norm{\tilde \eta}^2.
\end{aligned}
$$
By taking  
$\mu > 2 C_\eta/(\underline{q}\rho_\eta)$ and 
\begin{equation}\label{eq:kp_bound} 
    k_p > \max \left\{ 1,
    \frac{k}{\underline{m}}, 
    \frac{1}{2 \overline{c}_H} \left( C_y + \mu \frac{L_1^2}{2\underline{q}\rho_\eta}\right) \right\} \doteq k_p^\star,
 \end{equation}
we obtain $\dot V \le - \gamma \big( \norm{\tilde \eta}^2 + \norm{y}^2 + \norm{w}^2 \big)$ for some constant $\gamma > 0$. This ensures the global exponential stability of the equilibrium $(\etastar, 0, 0)$ of the system in the normal-form coordinates. As previously stated, the global exponential convergence carries over to the original $(x,\lambda)$ coordinates by virtue of the globally Lipschitz diffeomorphism $\Phi^{-1}$.
\end{proof}

Assumptions~\ref{ass:Lip_diffeo}, \ref{ass:L1}, and~\ref{ass:lip_l} are Lipschitz-type regularity conditions, which are required to establish the global exponential convergence results in Theorems~\ref{corollary:flcmo} and \ref{th:picmo_ges} over $\R^n$. To verify these assumptions globally, sufficient conditions are that $\nabla_x f$ and $\nabla_x h$ are globally bounded and Lipschitz; see Appendix~II for detailed derivations. 

However, in a practical setting, optimization algorithms operate within bounded regions. Given any bounded set of initial conditions, one can define a sufficiently large compact sublevel set of the Lyapunov function that completely encloses them. Since $f, h \in C^2$, the Lipschitz conditions are automatically satisfied by restriction on any such compact set. Consequently, by using the local Lipschitz constants on this set, one can determine a valid $k_p^\star$ that makes the Lyapunov derivative strictly negative. This guarantees that the compact set is forward-invariant, and the exact same proof steps hold to conclude semi-global exponential stability counterparts of Theorems~\ref{th:flcmo_ges} and~\ref{th:picmo_ges}. Furthermore, since the local Lipschitz constants are generally much smaller than their global counterparts, we can significantly reduce the conservativeness of the closed-form expression of $k_p^\star$, which is provided in Appendix~I.

Our derivation of explicit tuning bounds distinguishes this work from singular perturbation approaches like~\cite{allibhoy_anytime_2025}, which does not provide a specific lower bound for $k_p$. The result in~\cite{cmo24} differs structurally from Theorem~\ref{th:picmo_ges} as it imposes a lower bound on $k_i$ rather than an upper one.

\section{Further comments}\label{sec:further_comments}
In this section, we investigate the relationship between the PI dynamics in~\eqref{eq:picmo_zform} with continuous-time ALM \cite{bertsekas1996constrained}, the FL dynamics in~\eqref{eq:fl-cmo-general}, and PDGD.

ALM can be framed as a particular case of PI dynamics \eqref{eq:picmo_zform}. By applying the primal-dual flow to the quadratically augmented Lagrangian $\mathcal{L}_w(x,\lambda): \R^{n} \times \R^m \rightarrow \R$, defined as $\mathcal{L}_w(x,\lambda) \doteq f(x) + \lambda^\top h(x) + \tfrac{w}{2}\,\|h(x)\|^2$ yields the PI dynamics \eqref{eq:picmo_zform} under the choices $k_p = w$, $k_i = 1$, and $z = \lambda$. 
Previous works show that ALM converges locally for a sufficiently large $w$ under Assumption~\ref{ass:convexity}. However, under nonlinear constraints, convergence is limited to a neighborhood of the optimum. In contrast, we establish global exponential convergence for ALM as a corollary of Theorem~\ref{th:picmo_ges}, provided that the penalty parameter is chosen as $w \geq k_p^\star$.

Second, we note that the PI dynamics~\eqref{eq:picmo_lamb_form}  admits an interpretation 
as a gradient-flow approximation of the FL 
dynamics~\eqref{eq:fl-cmo-general}, yielding a more efficient implementation achieving the same optimal solution. The control law~\eqref{eq:fl-cmo_lambda} can be written equivalently as
\begin{equation}\label{eq:lambda_opt}
    \lambda(x) = \arg\min_{\sigma\in\R^m}\ \|A(x)\sigma-b(x)\|^2,
\end{equation}
where $A(x) \doteq \nabla_x^\top h(x) \nabla_x h(x)$ and $b(x)$ $\doteq$ $-\nabla_x^\top h(x)$ $\nabla_x f(x)$ $+$ $\mathcal{G}(h(x))$. Under Assumption~\ref{ass:Jrank}, $A(x) \succ 0$, and~\eqref{eq:lambda_opt} admits a unique minimizer. The negative gradient flow that asymptotically tracks this minimizer is
\begin{equation}\label{eq:sigma_gd}
    \dot{\sigma} = -P_0(x)\nabla_\sigma \|A(x)\sigma-b(x)\|^2,
\end{equation}
with $P_0 = P_0^\top \succ 0$ chosen arbitrarily. Choosing $\mathcal{G}(y) = -ky$ with $k \in \R^+$ and $P_0 = \frac{\alpha}{2}(\nabla_x^\top h(x) \nabla_x h(x))^{-1}$ recasts~\eqref{eq:sigma_gd} as $\dot\sigma = -\alpha (A(x)\sigma - b(x))$, which, with the identifications $k_p = \alpha$ and $\lambda = \sigma$, yields~\eqref{eq:picmo_lamb_form}. 

This interpretation resembles the 
construction in~\cite{allibhoy_anytime_2025} for variational inequalities. We also remark that solving~\eqref{eq:fl-cmo_lambda} is the most expensive step of the FL dynamics, as it requires solving a linear system of size $m$; the PI dynamics bypasses this inversion, and the approximation improves as $k_p = \alpha$ increases and the gradient flow becomes faster.

Finally, PDGD coincides with the pure integral action of PI ($k_p = 0$, $k_i = 1$): no choice of $\alpha>0$ and $k>0$ in the gradient-flow construction recovers it, since $k_p = \alpha = 0$ forces $k_i = k \alpha = 0$ for all $k \in \R$. Therefore, PDGD (whose global exponential stability has been established for strongly convex problems in~\cite{qu19}) may fail to converge even if the optimization problem satisfies the assumptions considered in this paper, as shown in the illustrative example in Sec. \ref{sec:academic_es}. 

\section{Illustrative Example} 
\label{sec:academic_es}
\begin{figure}
    \centering
    \includegraphics[width=\linewidth]{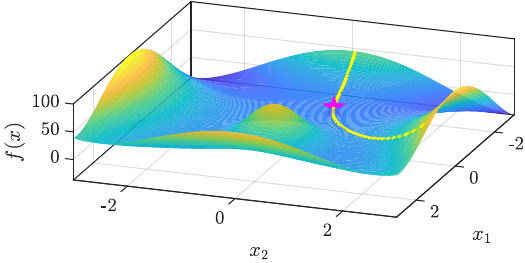}
    \caption{Graphical illustration of the problem in Sec.~\ref{sec:academic_es}.}
    \label{fig:es1_3d}
\end{figure}

\begin{figure}
    \centering
    \includegraphics[width=\linewidth]{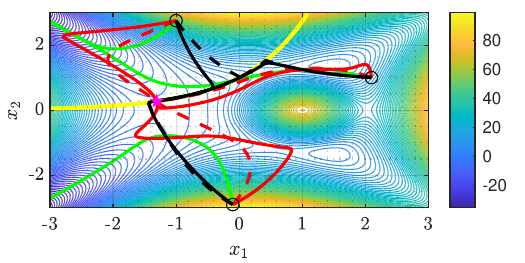}
    \caption{Trajectories generated by PI and FL for the problem in Sec.~\ref{sec:academic_es}. Level lines of $f(x)$ and global optimal solution ({\color{magenta}{$\star$}}). Trajectories generated by FL with $k=1$ (\tikzref[color=red, line width=1.5pt]), FL with $k=10$ (\tikzref[color=red, line width=1.5pt, dash pattern=on 4pt off 1pt]), PI with $k_p=100$ (\tikzref[color=black, line width=1.5pt]), PI with $k_p = k_p^\star \approx 1.94 \times 10^{15}$ (\tikzref[color=black, line width=1.5pt, dash pattern=on 4pt off 1pt]), and PDGD (\tikzref[color=green, line width=1.5pt]). All trajectories are initialized at points ($\circ$).}
    \label{fig:es1_traj}
\end{figure}
We consider a two-dimensional example to illustrate a scenario in which Assumption~\ref{ass:convexity} is satisfied despite $f(x)$ being non-convex, $h(x)$ being nonlinear, and the two functions being nontrivially related. Specifically
$$ f(x) \! = \!(x_1^2\!-\!x_2^2\!-\!1)^2\!+\!0.7(x_1^3\! + \!x_1x_2^2)\!+\!70e^{-2(x_1-1)^2-2x_2^2} $$
and $h(x) = x_2 - e^{x_1}$. The function $f(x)$ has three local minima, one local maximum, and one saddle point. Assumption~\ref{ass:Jrank} holds since $\nabla_x h(x) = [-e^{x_1},1]^\top \neq 0$ for all $x \in \R^2$. $f(x)$ is strongly convex when restricted to $\{x:h(x)=0\}$, thereby satisfying Assumption \ref{ass:convexity}.
Moreover, as discussed in Sec.~\ref{sec:PIalg}, the regularity 
Assumptions~\ref{ass:Lip_diffeo}, \ref{ass:L1}, and~\ref{ass:lip_l} are verified on any compact forward-invariant set containing the initial conditions. Consequently, in this specific illustrative example, the developed theoretical results hold in a semiglobal sense. Fig.~\ref{fig:es1_3d} represents the optimization problem data, and Fig.~\ref{fig:es1_traj} shows the trajectories generated by FL dynamics in Eq.~\eqref{eq:fl-cmo-general} and PI dynamics in Eq.~\eqref{eq:picmo_zform} for different initial conditions and $k_i = 1$. For the PI scheme, two values of $k_p$ are considered: the conservative threshold $k_p^\star$ provided by the explicit closed-form bound derived in Appendix~I on the compact set $\mathcal{K}$, and a smaller, less conservative value. We observe that, as expected from the theoretical results established in this work, all trajectories converge to the unique global optimal solution of the optimization problem. Conversely, the trajectories generated by PDGD, initialized at the same points 
as for~\eqref{eq:picmo_zform}, diverge.

\section{Conclusion}\label{sec:CON}
We establish global exponential convergence of the proportional-integral and feedback linearization dynamics for a class of equality-constrained non-convex optimization problems. While previous works assume strong convexity in the ambient space, this work extends the analysis to a class of non-convex problems that satisfy a structural geometric property induced by the equality constraints. Moreover, we show that the proportional-integral dynamics is a lower-complexity approximation of the feedback-linearization dynamics that yields the same optimal solution. 
Future work will focus on overconstrained optimization and discrete-time models.

\bibliographystyle{IEEEtran}
\bibliography{bibliog}

\appendix 

\section{Bound on the parameter $k_p$}
In this section, we provide an explicit bound on the parameter $k_p$ in terms of the problem's data.
\begin{lemma}[Explicit form of Lemma~\ref{lemm:bnd_r}]\label{lem:explicit_r}
Let $\overline m > 0$ be such that $\|\nabla_x^\top h(x) \nabla_x h(x)\| \leq \overline m$~for all $x \in \R^n$. Under Assumptions~\ref{ass:Jrank}, \ref{ass:Lip_diffeo}, and~\ref{ass:lip_l}, the bound 
in~\eqref{eq:r_bound} holds with
\begin{equation}\label{eq:c_r_explicit}
    \ell_r^\eta,\, \ell_r^y \leq L_r \doteq \overline m\, L_2\, L_\Psi.
\end{equation}
\end{lemma}

\begin{proof}
Given $\tilde x$ defined in Lemma~\ref{lemm:bnd_r} and $\varphi(x)$ in Assumption~\ref{ass:lip_l}, it holds
\begin{equation}
    \nabla_x^\top h(\tilde x) \nabla_x f(\tilde x) 
    = -\nabla_x^\top h(\tilde x) \nabla_x h(\tilde x)\, \varphi(\tilde x).
\end{equation}
Consider $r(\eta,y)$ defined in Lemma~\ref{lemm:bnd_r}. Substituting $\lstar = \varphi(\xstar)$ and multiplying the first term by $\nabla_x^\top h(\tilde x) \nabla_x h(\tilde x)\,\left(\nabla_x^\top h(\tilde x) \nabla_x h(\tilde x)\right)^{-1}$, we get
\begin{equation}
    r(\eta,y) 
    = \nabla_x^\top h(\tilde x) \nabla_x h(\tilde x)\, \big(\varphi(\xstar) - \varphi(\tilde x)\big).
\end{equation}
Taking norms and applying Assumption~\ref{ass:lip_l},
\[
    \|r(\eta,y)\| 
    \leq \overline m\, \|\varphi(\xstar) - \varphi(\tilde x)\| 
    \leq \overline m\, L_2\, \|\tilde x - \xstar\|.
\]
Finally, by Assumption~\ref{ass:Lip_diffeo},
\begin{align*}
    \|\tilde x - \xstar\| 
    &= \|\Phi^{-1}(y,\eta) - \Phi^{-1}(0,\etastar)\|
    \leq L_\Psi\,\norm{\begin{pmatrix}y \\ \eta-\etastar
    \end{pmatrix}} \\
    & \leq L_\Psi\,\big(\|y\| + \|\eta-\etastar\|\big),
\end{align*}
which yields the claim.
\end{proof}

\begin{lemma}[Explicit lower bound on $\overline\Psi$]\label{lem:lower_psi}
Under the setting and notation of the proof of 
Theorem~\ref{th:picmo_ges}, with the choice $\theta = 2\underline m/k$ 
and the condition $k_p \geq 2k/\underline m$, the matrix 
$\overline\Psi$ satisfies $\overline\Psi \succeq (\underline m/8)\, I$, 
i.e., the constant $\overline c_H$ in~\eqref{eq:Psi_bound} can be taken 
as $\overline c_H = \underline m/8$.
\end{lemma}

\begin{proof}
Under the choice $\theta = 2\underline m/k$, $\overline\Psi$ in the proof of Theorem~\ref{th:picmo_ges} reads
\begin{equation}
    \overline\Psi = \begin{bmatrix} H-\frac{k}{k_p}I & H - \underline m I \\ \star & H \end{bmatrix} 
\end{equation}
Let us consider the Schur block factorization of $\overline\Psi$:
\[
    \overline\Psi = T \begin{bmatrix} S & 0 \\ 0 & H \end{bmatrix} T^\top, 
    \qquad
    T = \begin{bmatrix} I & B H^{-1} \\ 0 & I \end{bmatrix},
\]
with 
\begin{equation}
    B \doteq H - \underline m I, \quad S = \Bigl(2\underline m - \frac{k}{k_p}\Bigr)I - \underline m^2\, H^{-1}.
\end{equation}
For $k_p \geq 2k/\underline m$, it holds $2\underline m - k/k_p \geq 3\underline m/2$, and since
$H \succeq \underline m I$ implies $\underline m^2 H^{-1} \preceq \underline m I$,
\begin{equation}\label{eq:bound_S_appendix}
    S \;\succeq\; \Bigl(\frac{3\underline m}{2} - \underline m\Bigr)I
      = \frac{\underline m}{2}\, I.
\end{equation}
For any $v \in \R^{2m}$, set $u = T^\top v$. Using $H \succeq \underline m I \succeq (\underline m/2)\, I$ and using Eq.~\eqref{eq:bound_S_appendix}
\begin{equation}\label{eq:appendix_lemmPsibnd_bnswithu}
    v^\top \overline\Psi v 
    = u^\top \begin{bmatrix} S & 0 \\ 0 & H \end{bmatrix} u 
    \geq \tfrac{\underline m}{2}\, \|u\|^2.
\end{equation}
We now bound $\norm{u}^2= \|T^\top v\|^2 \geq \|v\|^2/\|T^{-\top}\|^2$. Considering that
\begin{equation}
    T^{-1} = \begin{bmatrix} I & -B H^{-1} \\ 0 & I \end{bmatrix},
\end{equation}
and the decomposition $T^{-1} = I + N$, wth $N \doteq T^{-1} - I$. From $\underline m I \preceq H$ we have $0 \preceq \underline m H^{-1} \preceq I$, hence $\|B H^{-1}\| = \|I - \underline m H^{-1}\| \leq 1$ and therefore $\|N\| \leq 1$. The triangle inequality yields $\|T^{-1}\| \leq \|I\| + \|N\| \leq 2$, so $\|u\|^2 \leq \norm{v}^2/4$. Finally, using this bound on $\norm{u}^2$ in Eq.~\eqref{eq:appendix_lemmPsibnd_bnswithu} yields
\begin{equation}
    v^\top \overline\Psi v \geq (\underline m/8)\,\|v\|^2,
\end{equation}
which proves $\overline\Psi \succeq (\underline m/8)\, I$.
\end{proof}

We now use Lemmas~\ref{lem:explicit_r} and~\ref{lem:lower_psi} to 
derive an explicit form of the threshold $k_p^\star$ in 
Theorem~\ref{th:picmo_ges}.

\begin{theorem}[Explicit threshold]\label{th:explicit_threshold}
Let the assumptions of Theorem~\ref{th:picmo_ges} and Lemmas~\ref{lem:explicit_r}-\ref{lem:lower_psi} hold. Then, the conclusions of 
Theorem~\ref{th:picmo_ges} hold whenever
\begin{equation}\label{eq:explicit_kp}
    k_p \geq \max\left\{
        1,\ \frac{2k}{\underline m},\ \kappa
    \right\},
\end{equation}
where
\begin{equation}\label{eq:kappa}
    \kappa \doteq 
    \frac{12 L_r}{\underline m} 
    + \frac{64 L_r^2}{\underline m^2}
    + \frac{8 L_1^2\, L_r}{\underline{q}^2\rho_\eta^2\, \underline m}
    + \frac{128 L_1^2\, L_r^2}{\underline{q}^2\rho_\eta^2\, \underline m^2},
\end{equation}
together with $k_i \leq \underline m\, k_p^2/2$.
\end{theorem}

\begin{proof}
We follow the proof of Theorem~\ref{th:picmo_ges} and make explicit 
each of the constants $\overline c_H, \mu, C_y, C_\eta$ defined 
therein. By Lemma~\ref{lem:lower_psi}, with the choice 
$\theta = 2\underline m/k$ and the condition 
$k_p \geq 2k/\underline m$ (equivalently, 
$k_i \leq \underline m k_p^2/2$), the constant $\overline c_H$ 
in~\eqref{eq:Psi_bound} can be taken as 
$\overline c_H = \underline m/8$. From Lemma~\ref{lem:explicit_r}, 
$\ell_r^\eta, \ell_r^y \leq L_r$. Substituting into the definitions 
of $C_y$ and $C_\eta$, we have
\begin{align}
    C_y \leq 3 L_r + \frac{16 L_r^2}{\underline m}, \quad
    C_\eta \leq L_r + \frac{16 L_r^2}{\underline m}.
\end{align}
Choosing $\mu = 4 C_\eta/(\underline{q} \rho_\eta)$ (which satisfies the 
condition $\mu > 2 C_\eta/(\underline{q} \rho_\eta)$ strictly), the threshold 
in~\eqref{eq:kp_bound} becomes
\begin{equation}
    k_p \geq \frac{4}{\underline m} 
        \left( C_y + \mu \frac{L_1^2}{2 \underline{q} \rho_\eta} \right)
    = \frac{4 C_y}{\underline m} + \frac{8 C_\eta L_1^2}{\underline m \underline{q}^2 \rho_\eta^2},
\end{equation}
which, after substitution of the bounds on $C_y$ and $C_\eta$, yields~\eqref{eq:explicit_kp} with~\eqref{eq:kappa}.
\end{proof}

\section{Sufficient conditions for Assumptions 4, 5, and 6}
In this section, we verify the sufficient conditions on $f$ and $h$ so that Assumptions 4,5 and 6 are globally satisfied.
\begin{proposition}[A sufficient condition for Assumption~\ref{ass:Lip_diffeo}]
\label{prop:app2}
    Assume that $\norm{\nabla_x h(x)} \leq B_h$ and $\norm{\nabla_x q(x)} \leq B_q$ for some $B_h,B_q > 0$. Then, $\Phi$ in Eq.~\eqref{eq:diffeo} is globally Lipschitz.
    Moreover, if Assumption~\ref{ass:Jrank} is satisfied, then also the inverse map $\Phi^{-1}$ is globally Lipschitz.
\end{proposition}
\begin{proof}
    Under the considered assumptions, the Jacobian of $\Phi = [h(x)^\top,\, q(x)^\top]^\top$ satisfies
    \begin{equation*}
        \|\nabla_x^\top \Phi(x)\|^2 \leq \|\nabla_x h(x)\|^2 + \|\nabla_x q(x)\|^2 \leq B_h^2 + B_q^2.
    \end{equation*} 
    By the mean-value inequality,
    \begin{equation*}
    \|\Phi(x) - \Phi(y)\|  \leq  \sup_{z}\|\nabla_x^\top \Phi(z) \|\,\|x - y\|  \leq  L_\Phi \,\|x - y\|,
    \end{equation*} 
    for all $L_\Phi \leq \sqrt{B_h^2+B_q^2}$.\\
    For $\Phi^{-1}$, we use the inverse function theorem and the block-inverse formula established in the proof of Theorem~\ref{th:flcmo_ges}. Given $x = \Phi^{-1}(\xi)$, we have
    \begin{align*}
        \nabla_\xi^\top (\Phi^{-1})(\xi) &= (\nabla_x^\top \Phi(x))^{-1}  = \\
        &=\left[ \nabla_x h(x)\,H_i(x) ~~ \nabla_x q(x) Q(x)^{-1}\right]
    \end{align*}
    where $H_i(x) \doteq \left(\nabla_x^\top h(x) \nabla_x h(x) \right)^{-1}$. Then, considering the bounds on $Q(x), \|\nabla_x q(x)\|$ and that Assumption~\ref{ass:Jrank} implies the bound $\|H_i(x)\| \le 1/\underline{m}$, we obtain
    \begin{align*}
        & \|\nabla_\xi^\top (\Phi^{-1})(\xi) \| \leq \|\nabla_x h(x)\|\,\|H_i(x)\| + \\
        &\quad + \|\nabla_x q(x)\| \norm{Q(x)^{-1}} \leq \frac{B_h}{\underline{m}} + \frac{B_q}{\underline{q}}.
    \end{align*}
    Hence, by the mean-value inequality, $\Phi^{-1}$ is globally Lipschitz with constant $L_\Psi \leq B_h/\underline{m}+{B_q}/{\underline{q}}$.
\end{proof}

\begin{proposition}[Sufficient conditions for Assumption~\ref{ass:L1}]
    Let the assumptions in Proposition~\ref{prop:app2} hold and assume $\nabla_x f(x)$ is globally Lipschitz with constant $L_f$. Furthermore, let Assumptions~\ref{ass:Jrank} and \ref{ass:Lip_diffeo} hold. Then, there exists a constant $L_1>0$ such that the Lipschitz property~\eqref{eq:transverse_bound} holds.
\end{proposition}
\begin{proof}
    It is trivially shown that, under Assumption~\ref{ass:Jrank}, the orthogonal projector onto the null space of $\nabla_x^\top h(x)$ is globally Lipschitz whenever $\nabla_x^\top h(x)$ is globally Lipschitz. By standard perturbation results for invariant subspaces (see, e.g.~\cite{kato1966perturbation}), a Lipschitz family of orthogonal projectors with constant rank admits a Lipschitz orthonormal basis of its image. Therefore, $\nabla_x q(x)$ is globally Lipschitz with some constant $L_q$, i.e., 
    \begin{equation*}
       \| \nabla_x q(x_1) - \nabla_x q(x_2) \| \leq L_q \| x_1 - x_2 \|, \quad \forall x_1,x_2 \in \R^n.
    \end{equation*}
    Next, consider $C(\tilde x, p) \doteq \nabla_x^\top q(\tilde x) \nabla_x f(\tilde x) - \nabla_x^\top q(p) \nabla_x f(p)$, where $\tilde x = \Phi^{-1}(y,\eta)$ and $p = \Phi^{-1}(0,\eta)$. Using the triangle inequality:
    \begin{align}\label{eq:appendix_suff_ass4_defC}
        C(\tilde x, p) = D_q(\tilde x, p) \nabla_x f(\tilde x) + \nabla_x^\top q(p) D_f(\tilde x, p),
    \end{align}
    with $D_q(\tilde x, p) = \nabla_x^\top q(\tilde x) - \nabla_x^\top q(p)$ and  $D_f(\tilde x, p) = \nabla_x f(\tilde x) - \nabla_x f(p)$.
    We bound each term of $C(\tilde x, p)$ in Eq.~\eqref{eq:appendix_suff_ass4_defC} separately: 
    \begin{itemize}
        \item for the first one, we have
        \begin{equation}
            \|D_q(\tilde x, p) \nabla_x f(\tilde x) \| \leq L_q\,B_f\,\|\tilde x - p\|
        \end{equation}
        by using the global bound $\|\nabla_x f(\tilde x)\| \le B_f$ and the Lipschitz property of $\nabla_x q(x)$. 
        \item for the second one, we have
        \begin{equation}
            \| \nabla_x^\top q(p)\, D_f(\tilde x, p) \| \leq B_q L_f\,\|\tilde x - p\|
        \end{equation}
        by using the Lipschitz property of $\nabla f(x)$ and $\|\nabla_x q(x)\| \leq B_q$.
    \end{itemize}
    Adding the two bounds, we get
    \begin{equation}\label{eq:ass4_sufficient_quasi}
        \norm{ C(\tilde x, p) } \leq (B_q L_f + L_q B_f) \norm{\tilde x - p}
    \end{equation}
    Finally, since \begin{equation}\label{eq:xtilde_p_y}
    \|\tilde x - p\| = \|\Phi^{-1}(y,\eta) - \Phi^{-1}(0,\eta)\|     \leq L_\Psi\,\|y\|
    \end{equation}
    by Assumption~\ref{ass:Lip_diffeo}, we obtain the result with $L_1 = L_\Psi(L_q B_f + B_q L_f)$ by combining Eqs.~\eqref{eq:ass4_sufficient_quasi} and \eqref{eq:xtilde_p_y}.
\end{proof}

\begin{proposition}[Sufficient conditions for Assumption~\ref{ass:lip_l}]
\label{prop:global_sufficient}
    Let Assumption~\ref{ass:Jrank} hold, and suppose there exist constants
    $B_f, B_h > 0$ such that, for all $x \in \R^n$, $\|\nabla_x f(x)\| \leq B_f$, $\|\nabla_x h(x)\| \leq B_h$. Further assume that $\nabla_x f(x), \nabla_x h(x)$ are Lipschitz continuous with constants $L_f$ and $L_h$, respectively.
    Then, Assumption~\ref{ass:lip_l} holds with
    \begin{equation}\label{eq:L2_explicit}
      L_2 \leq  \frac{B_h M_f + M_h B_f}{\underline{m}} + \frac{2\,B_h^2 M_h B_f}{\underline{m}^2}.
    \end{equation}
\end{proposition}
 
\begin{proof}
    Define $H(x) \doteq \nabla_x^\top h(x) \nabla_x h(x)$. By Assumption~\ref{ass:Jrank} we have $\underline{m} I \preceq H(x)$ and by $\|\nabla_x h(x)\| \le B_h$ it holds $\|H(x)\|\le B_h^2$. For arbitrary $x_1,x_2\in\R^n$, we can write
    \begin{subequations} \begin{align}
        &\varphi(x_1)-\varphi(x_2) = \\
        &= -\bigl[H(x_1)^{-1}\!-\!H(x_2)^{-1}\bigr]\nabla_x^\top h(x_2) \nabla_x f(x_2) + \label{eq:prop_suff6_diffphi1} \\
        &-\! H(x_1)^{-1}\bigl[\nabla_x^\top h(x_1) \nabla_x f(x_1) \!-\! \nabla_x^\top  h(x_2) \nabla_x f(x_2)\bigr]. \label{eq:prop_suff6_diffphi2}
    \end{align}\end{subequations}
    For the term in Eq.~\eqref{eq:prop_suff6_diffphi1}, we use the resolvent identity
    $H(x_1)^{-1}-H(x_2)^{-1}=H(x_1)^{-1}(H(x_2)-H(x_1))H(x_2)^{-1}$ to obtain
    \begin{equation}
        \|H(x_1)^{-1}-H(x_2)^{-1}\| \leq \frac{\|H(x_1)-H(x_2)\|}{\underline{m}^2}.
    \end{equation}
    To bound $\|H(x_1)-H(x_2)\|$, we add and subtract
    $\nabla_x^\top h(x_1) \nabla_x h(x_2)$, obtaining:
    \begin{equation}\begin{aligned}
        H(x_1)-&H(x_2) = \nabla_x^\top h(x_1) \bigl[\nabla_x h(x)-\nabla_x h(x_2)\bigr] +\\
        +&\bigl[\nabla_x h(x_1)-\nabla_x h(x_2)\bigr]^\top \nabla_x h(x_2).
    \end{aligned}\end{equation}
    Then, using $\|\nabla_x h(x_1)\|,\,\|\nabla_x h(x_2)\|\le B_h$, we get the bound
    \begin{subequations}\begin{align}
        \|H(x_1)-H(x_2)\| &\leq 2\,B_h\,\|\nabla_x h(x_1)-\nabla_x h(x_2)\| \\ &\leq 2\,B_h\,L_h\,\|x_1-x_2\|.
    \end{align}\end{subequations}
    Next, using $\|\nabla_x^\top h(x_2) \nabla_x f(x_2)\|\leq B_h B_f$, we conclude that the term in Eq.~\eqref{eq:prop_suff6_diffphi1} is Lipschitz with constant   $({2\,B_h^2\,M_h\,B_f})/{\underline{m}^2}$.
    Next, we consider the term in Eq.~\eqref{eq:prop_suff6_diffphi2}. Using the equality
    \begin{subequations}
    \begin{align}
        &\nabla_x^\top h(x_1) \nabla_x f(x_1) -\nabla_x^\top h(x_2) \nabla_x f(x_2) =\\
        &=\nabla^\top_x h(x_1) D_f(x_1,x_2) +D_h(x_1,x_2) \nabla_x f(x_2).
    \end{align}\end{subequations}
    where 
    \begin{subequations}\begin{align}
        D_f(x_1,x_2) &\doteq \nabla_x f(x_1) - \nabla_x f(x_2)\\
        D_h(x_1,x_2) &\doteq \nabla_x^\top h(x_1) - \nabla^\top_x h(x_2),
    \end{align} \end{subequations}
    with 
    \begin{subequations}\begin{align}
        \norm{D_f(x_1,x_2)} \leq L_f \norm{x_1-x_2}, \\ \norm{D_f(x_1,x_2)} \leq L_h \norm{x_1-x_2},
    \end{align}\end{subequations}
    we obtain
    \begin{equation} \begin{aligned}
        &\|\nabla_x^\top h(x_1) \nabla_x f(x_1)-\nabla_x^\top h(x_2) \nabla_x f(x_2)\| \leq\\
        &(B_h M_f+M_h B_f)\|x_1-x_2\|.
    \end{aligned} \end{equation}
    Considering also the bound $\|H(x)^{-1}\| \leq 1/\underline{m}$ the term in Eq.~\eqref{eq:prop_suff6_diffphi2} is Lipschitz with constant $({B_h M_f+M_h B_f})/({\underline{m}})$.
    Adding the constants on the two terms yields~\eqref{eq:L2_explicit}.
\end{proof}

\begin{remark}
The constants $L_1$ and $L_2$ in
Proposition~\ref{prop:global_sufficient} are \emph{explicitly depending on the
problem data} $B_f, B_h, L_f, L_h, \underline{m}$, consistently with
the explicit threshold $k_p^\star$ derived in
Theorem~\ref{th:explicit_threshold}: one may substitute these
expressions into~\eqref{eq:kappa} to obtain $\kappa$ purely in terms
of first and second derivatives of $f$ and $h$.
\end{remark}

\end{document}